\documentclass[12pt]{article}
\usepackage{amssymb}
\usepackage{amsthm}
\usepackage{amsmath, mathrsfs}
\usepackage{color}
\usepackage{hyperref}
\usepackage{enumerate}
\usepackage{latexsym,array}
\usepackage{amsfonts}
\usepackage{shadow}
\newcommand{\be}{\begin{equation}}
\newcommand{\ee}{\end{equation}}
\newcommand{\bea}{\begin{eqnarray*}}
\newcommand{\eea}{\end{eqnarray*}}
\newcommand{\ba}{\begin{array}}
\newcommand{\ea}{\end{array}}
\newcommand{\bi}{\begin{itemize}}
\newcommand{\ei}{\end{itemize}}
\newcommand{\bc}{\begin{center}}
\newcommand{\ec}{\end{center}}
\newcommand{\bfr}{\begin{flushright}}
\newcommand{\efr}{\end{flushright}}

\newtheorem{Pa}{Paper}[section]
\newtheorem{theorem}[Pa]{{\bf Theorem}}

\newtheorem{Dn}[Pa]{{\bf Definition}}
\newtheorem{Cy}[Pa]{{\bf Corollary}}

\newtheorem{Rk}[Pa]{{\bf Remark}}

\begin{document}

\title{Approximation by polynomials on quaternionic compact sets}
\author{Sorin G. Gal\\
University of Oradea\\
Department of Mathematics\\ and Computer Science\\
Str. Universitatii Nr. 1\\
410087 Oradea,
Romania \\
galso@uoradea.ro
\and Irene Sabadini\\
Dipartimento di Matematica\\
Politecnico di Milano\\
Via Bonardi 9\\
20133 Milano, Italy\\
irene.sabadini@polimi.it}

\date{}
\maketitle

\begin{abstract}
In this paper we obtain several extensions to the quaternionic setting of some results concerning the approximation by polynomials of functions continuous on a compact set and holomorphic in its interior. The results include approximation on compact starlike sets and compact axially symmetric sets. The cases of some concrete particular sets are described in details, including quantitative estimates too.
\end{abstract}

\textbf{AMS 2010 Mathematics Subject Classification}: Primary 30G35; Secondary 30E10, 41A25.

\textbf{Keywords and phrases}: Mergelyan's theorem, quaternions, Riemann mapping, axially symmetric sets, approximation by polynomials, convolution operators, Cassini pseudo-metric, Cassini cell, order of approximation, slice regular functions.

\section{Introduction}

It is well-known the fact that the Mergelyan's approximation theorem is the ultimate development and generalization of the Weierstrass approximation theorem and Runge's theorem in the complex plane. It can be stated as follows (see \cite{Merge}):

\begin{theorem}\label{Mergelyan}
Let $K$ be a compact subset of the complex plane $\mathbb{C}$ such that $\mathbb{C}\setminus K$ is connected. Then, every continuous function on $K$, $f : K\to \mathbb{C}$, which is holomorphic in the interior of $K$, can be approximated uniformly on $K$ by polynomials.
\end{theorem}

Notice that all the known proofs of this result, based on the methods in complex analysis use the Riemann mapping theorem.

The main goal of the present paper is to extend the Mergelyan's result to the case of slice regular functions of quaternionic variable.
But the attempt to extend this result into its full generality encounters some obstacles which necessarily appear in the proof: firstly,
 a Riemann mapping theorem in the quaternionic setting is available only for a particular class of sets that we denote by $\mathfrak R (\mathbb H)$, as we will show in this paper. However, its validity is unknown in the general case.
 Another obstacle is the fact that the composition  of two slice regular functions
is not necessarily a slice regular function. Consequently, the Mergelyan's theorem in its full generality seems to be
not possible in the quaternionic setting. However, the case in which the Riemann mapping theorem is true, in the quaternionic case, is precisely the case in which the composition of slice regular functions gives a slice regular function.
We will then show that for the class of starlike sets and for sets in $\mathfrak R (\mathbb H)$, it is possible to get an approximation by polynomials in the spirit of Mergelyan's result.

The plan of the paper goes as follows. In Section 2 we present without proofs some concepts and results which will
be useful for the next sections. In Section 3 we discuss the Riemann mapping theorem in the case of sets in $\mathbb{C}$ symmetric with respect to the real axis
  and we prove it for a subclass of axially symmetric sets in  $\mathbb{H}$. In Section 4 firstly we obtain uniform approximation results by polynomials in the cases of starlike sets and axially symmetric sets. Then, by using convolutions with the de la Vall\'ee-Poussin trigonometric kernel, two quantitative approximation results in the cases of compact balls with centers real numbers and of Cassini-cells are proved. Also,
  we point out that by using other well known trigonometric kernels (different from that of the de la Vall\'ee-Poussin), in these last two cases we obtain even essentially better quantitative estimates, in terms of $\omega_{p}(f ; 1/n)$, $p\ge 2$ and $E_{n}(f)$-the best approximation quantity. The section ends with an uniform approximation result on the boundary of open sets in $\mathfrak R(\mathbb H)$ by polynomials in $q$ and $q^{-1}$.

Note that everywhere in the paper the symbol $\subset $ will denote strict inclusion and when we allow equality we will use the symbol $\subseteq$.

\section{Preliminary results}
The noncommutative field $\mathbb{H}$ of quaternions consists of elements of the form
$q=x_0  + x_{1} i +x_{2} j + x_{3} k$, $x_i\in \mathbb{R}$, $i=0,1, 2,3$,
where the imaginary units $i, j, k$ satisfy
$$i^2=j^2=k^2=-1, \ ij=-ji=k, \ jk=-kj=i, \ ki=-ik=j.$$ The real number $x_0$ is called real part of $q$
while $ x_{1} i +x_{2} j + x_{3} k$ is called imaginary part of $q$.
We define the norm of a quaternion $q$ as
$\|q\|=\sqrt{x_0^2+x_1^2 +x_2^2+x_3^3}$.
By
$\mathbb B$ we denote the open unit ball in $\mathbb H$, i.e.
 $$\mathbb{B}=\{q = x_0+i x_{1} + j x_{2} + k x_{3}, \mbox{ such that } x_0^2+x_{1}^{2}+x_{2}^{2}+x_{3}^{3}<1\}$$
while by  $\mathbb{S}$ we denote the unit sphere of purely imaginary quaternion, i.e.
$$\mathbb{S}=\{q = i x_{1} + j x_{2} + k x_{3}, \mbox{ such that } x_{1}^{2}+x_{2}^{2}+x_{3}^{3}=1\}.$$

Note that if $I\in \mathbb{S}$, then $I^{2}=-1$.  For any fixed $I\in\mathbb{S}$ we define $\mathbb{C}_I:=\{x+Iy; \ |\ x,y\in\mathbb{R}\}$, which can be identified with a complex plane.
Obviously, the real axis belongs to $\mathbb{C}_I$ for every $I\in\mathbb{S}$.
Any non real quaternion $q$ is uniquely associated to the element $I_q\in\mathbb{S}$
defined by $I_q:=( i x_{1} + j x_{2} + k x_{3})/\|  i x_{1} + j x_{2} + k x_{3}\|$ and so $q$ belongs to the complex plane $\mathbb{C}_{I_q}$.

The functions we will consider in this paper are the so called slice regular functions of a quaternion variable first introduced in \cite{gs}. For a treatment of these functions and their applications we refer the interested reader to \cite{CSS}.
\begin{Dn}\label{slice regular}
 Let $U$ be an open set
$\mathbb{H}$. A real differentiable
function $f:U \to \mathbb{H}$ is said to be slice regular if,
for every $I \in \mathbb{S}$, its restriction $f_I$ of $f$ to the complex
plane $\mathbb C_I$ satisfies
$$
\overline{\partial}_If(x+Iy)=\frac{1}{2}\Big(\frac{\partial}{\partial x}
+I\frac{\partial}{\partial y}\Big)f_I(x+Iy)=0.
$$
 The set of slice regular functions on $U$ will be denoted by $\mathcal{R}(U)$.
\end{Dn}
\begin{Rk}{\rm
The class of slice regular functions contains converging power series, thus also polynomials, of the variable $q$ and with quaternionic coefficients written  on the right.
In this paper,  the terminology "polynomial" will always refer to polynomials of the form $\sum_{k=0}^n q^ka_k$, i.e. with coefficients on the right.
}
\end{Rk}
We introduce a suitable definition of derivative:
\begin{Dn}
Let $U$ be an open set in $\mathbb{H}$, and let $f:U \to \mathbb{H}$
be a slice regular function. The slice derivative $\partial_s f$ of $f$,
is defined by:
\begin{displaymath}
\partial_s(f)(q) = \left\{ \begin{array}{ll}
\partial_I(f)(q)=\frac{1}{2}\Big(\frac{\partial}{\partial x}
+I\frac{\partial}{\partial y}\Big)f_I(x+Iy) & \textrm{ if $q=x+Iy$, \ $y\neq 0$},\\ \\
\displaystyle\frac{\partial f}{\partial x} (x) & \textrm{ if\  $q=x\in\mathbb{R}$.}
\end{array} \right.
\end{displaymath}
\end{Dn}
The definition of slice derivative is well posed because it is applied
only to slice regular functions. Moreover
$$
\frac{\partial}{\partial x}f(x+Iy)=
-I\frac{\partial}{\partial y}f(x+Iy)\qquad \forall I\in\mathbb{S},
$$
and therefore, analogously to what happens in the complex case,
$$ \partial_s(f)(x+Iy) =
\partial_I(f)(x+Iy)=\frac{\partial}{\partial x}(f)(x+Iy).
$$
For simplicity, we will write $f'(q)$ instead of $\partial_s(f)(q)$.
\\
Among the open sets in $\mathbb H$ there are some that are important since they are the domains on which slice regular functions possess nice properties.
\begin{Dn}
Let $U \subseteq \mathbb H$. We say that $U$ is
\textnormal{axially symmetric} if, for all $x+Iy \in U$,  all the elements $x+\mathbb{S}y=\{x+Jy\ | \ J\in\mathbb{S}\}$ are contained in $U$.
We say that $U$ is a {\em slice domain} (or s-domain for short) if it is a connected set whose intersection with every complex plane $\mathbb C_I$ is connected.
\end{Dn}
The following result shows that the values of a slice regular function defined on an axially symmetric s-domain can be computed if the values of a restriction to a complex plane are known:
\begin{theorem}[Representation Formula]\label{Repr_formula} Let
$f$ be a slice regular function defined an axially symmetric s-domain $U\subseteq  \mathbb{H}$. Let
$J\in \mathbb{S}$ and let $x\pm Jy\in U\cap\mathbb C_J$.  Then the following equality holds for all $q=x+Iy \in U$:
\begin{equation}\label{distribution_mon}
\begin{split}
f(x+Iy) &=\frac{1}{2}\Big[   f(x+Jy)+f(x-Jy)\Big]
+I\frac{1}{2}\Big[ J[f(x-Jy)-f(x+Jy)]\Big]\\
&= \frac{1}{2}(1-IJ) f(x+Jy)+\frac{1}{2}(1+IJ) f(x-Jy).
\end{split}
\end{equation}
\end{theorem}
The Representation Formula yields:
\begin{Cy}\label{ext}
Let $\Omega_J$ be a domain in $\mathbb C_J$ symmetric with respect to the real axis and such that $\Omega_J\cap\mathbb R\not=\emptyset$. Let $U$ be the axially symmetric s-domain defined by
$$
U=\bigcup_{x+Jy\in\Omega_J,\ I\in\mathbb S} \{x+Iy\}.
$$
If $f:\Omega_J\to \mathbb H$ satisfies $\overline{\partial}_J f=0$ then the function
$$
{\rm ext}(f)(x+Iy)=\frac{1}{2}\Big[   f(x+Jy)+f(x-Jy)\Big]
+I\frac{1}{2}\Big[ J[f(x-Jy)-f(x+Jy)]\Big]
$$
is the unique slice regular extension of $f$ to $U$.
\end{Cy}

\begin{Dn}\label{axial completion}
Let $\Omega_J$ be any open set in $\mathbb C_J$ and let
\begin{equation}
U=\bigcup_{x+Jy\in\Omega_J,\ I\in\mathbb S} \{x+Iy\}.
\end{equation}
We say that $U$ is the axially symmetric completion of $\Omega_J$ in $\mathbb H$.
\end{Dn}
\begin{Rk}{\rm
The Representation formula implies that
if $U\subseteq\mathbb H$ is an axially symmetric s-domain and $f\in\mathcal R(U)$ then
$f(q)=f(x+Iy)=\alpha(x,y)+I\beta(x,y),$
 where $\alpha(x,y)=\frac{1}{2}[   f(x+Jy)+f(x-Jy)]$, $\beta=\frac{1}{2}J[   f(x-Jy)-f(x-Jy)]$ do not depend on $J\in\mathbb S$. So $\alpha$,  $\beta$ are $\mathbb{H}$-valued differentiable functions, $\alpha(x,y)=\alpha(x,-y)$, $\beta(x,y)=-\beta(x,-y)$ for all $x+Iy\in U$ and moreover satisfy the Cauchy-Riemann system.
 }
 \end{Rk}
 Thus, in view of the previous remark, slice regular functions are a subclass of the set of functions defined below (see \cite{gp1} for a treatment in a more general setting):
  \begin{Dn}\label{slice function}
  Let  $U\subseteq\mathbb H$ be an axially symmetric open set.
  Functions of the form $f(q)=f(x+Iy)=\alpha(x,y)+I\beta(x,y)$, where $\alpha$ $\beta$ are continuous  $\mathbb H$-valued functions such that $\alpha(x,y)=\alpha(x,-y)$, $\beta(x,y)=-\beta(x,-y)$ for all $x+Iy\in U$ are called {\em continuous slice functions}.
 \end{Dn}
 \begin{Rk}\label{reprform_continuous}{\rm
 \begin{enumerate}
 \item[1)] Continuous slice functions satisfy the Representation formula (\ref{distribution_mon}), see \cite{gp1}, section 3.3.
 \item[2)] Let  $U\subseteq\mathbb H$ be an axially symmetric open set.
  Functions of the form $f(q)=f(x+Iy)=\alpha(x,y)+I\beta(x,y)$, where $\alpha$ $\beta$ are continuously differentiable, satisfy the Cauchy-Riemann system and such that $\alpha(x,y)=\alpha(x,-y)$, $\beta(x,y)=-\beta(x,-y)$ for all $x+Iy\in U$ are called {\em slice regular} according to the terminology in \cite{gp1}. As a consequence of the Representation Formula (\ref{distribution_mon}), all the slice regular functions according to Definition \ref{slice regular} which are defined on axially symmetric s-domains are also slice regular according to \cite{gp1} and thus they are of the form $f(q)=f(x+Iy)=\alpha(x,y)+I\beta(x,y)$.
 \end{enumerate}
 }
 \end{Rk}
A subclass, denoted by the letter $\mathcal{N}$ (see e.g. \cite{CSS}, p. 152, Definition 4.11.2), of slice regular functions on an open set $U$ is defined as follows:
$$
 \mathcal{N}(U)=\{ f \ {\rm slice\ regular\ in}\ U\ :  \ f(U\cap \mathbb{C}_I)\subseteq  \mathbb{C}_I,\ \  \forall I\in \mathbb{S}\}.
$$
The class  $\mathcal{N}(U)$
 includes all elementary transcendental functions.\\
If one considers a ball $B(0,R)$ with center at the origin, it is immediate that a function slice regular on the ball belongs to $\mathcal{N}(B(0,R))$ if and only if its
power series expansion has real coefficients. Such functions are said to be {\em real}. More in general, if $U$ is an axially symmetric s-domain, then  $f\in\mathcal{N}(U)$ if and only if $f(q)=f(x+Iy)=\alpha(x,y)+I\beta(x,y)$
with $\alpha$, $\beta$ real valued, see \cite{CSremarks}, Proposition 2.1 and \cite{GMP}, Lemma 6.8.
\\
If $U$ is axially symmetric and if we denote by $\overline{q}$ the conjugate of a quaternion $q$, it can be shown that a function $f$ belongs to $\mathcal{N}(U)$
if and only is it satisfies $f(q)=\overline{f(\bar q)}$.
Note that complex functions defined on open sets $G\subset \mathbb{C}$ symmetric with respect to the real axis and such that $\overline{f(\bar z)}=f(z)$ are called in the literature intrinsic, see \cite{rinehart}.
In analogy with the complex case, we say that the functions in  the class $\mathcal N$ are  {\em quaternionic intrinsic}.

\begin{Rk}{\rm
As we shall see, quaternionic intrinsic functions will play an important role in this paper. At the moment, we just recall that the composition of slice regular functions, when defined, does not give a slice regular function, in general. However, if $f\in\mathcal N(U)$, $g\in\mathcal R(V)$ and  $f(U)\subseteq V$ then the composition $g(f(q))$  is slice regular.
}
\end{Rk}
We end this section by recalling
the following consequence of Runge's theorem (see Theorem 4.11 in \cite{runge}) which allows to approximate a slice regular function with polynomials. In the statement
$\overline{\mathbb{C}}_I$ denotes the extended complex plane ${\mathbb{C}}_I$.
\begin{theorem}\label{rungepoly}
Let $K$ be an axially symmetric compact set such that $\overline{\mathbb{H}}\setminus K$ is connected and such that $\overline{\mathbb{C}}_I\setminus (K\cap\mathbb{C}_I)$ is connected for all $I\in\mathbb{S}$. Let $f$ be slice regular in the open set $\Omega$ with $\Omega\supset K$. Then there exists a sequence $\{P_n\}$ of polynomials such that $P_n(q)\to f(q)$ uniformly on $K$.
\end{theorem}

\section{Riemann mappings for axially symmetric sets}
In this section we deal with the generalization  to the quaternionic setting of the famous Riemann mapping theorem which is of great importance, since it enters in the proof of several approximation results.
Let us begin by recalling the result in the complex plane :
\begin{theorem}[Riemann Mapping Theorem] Let $G\subset \mathbb C$ be a simply connected domain,  $z_0\in G$ and let  $\mathbb{D}=\{z\in \mathbb{C} \, :\,  |z|<1\}$ denote the open unit disk. Then there exists a unique bijective analytic function
$f:\ G\to\mathbb D$ such that $f(z_0)=0$, $f^{\prime}(z_0)>0$.
\end{theorem}
\begin{Rk}{\rm
The theorem holds more general for simply connected open subsets of the Riemann sphere which both lack at least two points of the sphere.
}
\end{Rk}
Our purpose is to generalize this theorem to the case of simply connected domains in $\mathbb H$. The problem is to characterize those open sets which can be mapped bijectively onto the unit ball of $\mathbb H$ by a slice regular function $f$ satisfying additional conditions prescribing its value and the value of its derivative at one point.

As we shall see, we have an answer which works for the class of  axially symmetric  s-domains which are simply connected that is, in practice, the class of all sets of our interest.
\\

     In \cite{duren}  the author defines functions which are {\em typically real} as those functions defined on the open unit disc $\mathbb D$ which are univalent and take real values just on the real line.  These functions have real coefficients when expanded into power series and so they are (complex) intrinsic. The image of such mappings is symmetric with respect to the real line, see  \cite{duren}, p.55. We have the following result which is basically already known, but we write the proof for the reader's convenience:

\begin{theorem}\label{Tm:Riemann intrinsic} Let $G\subset\mathbb C$, $G$ nonempty, be a simply connected domain such that $G\cap\mathbb R\not=\emptyset$.
For a fixed $x_0\in G\cap \mathbb{R}$, there  exists a unique bijective,  analytic function
$f:\ G\to\mathbb D$ with $f(x_0)=0$, $f'(x_0)>0$; then $f$ is  such that $f^{-1}$ is typically real
 if and only if $G$ is symmetric with respect to the real axis.
\end{theorem}
  \begin{proof} The map $f:\ G\to\mathbb D$ and such that $f(x_0)=0$, $f'(x_0)>0$ obviously exists by the Riemann mapping theorem. If  $G$ is symmetric with respect to the real axis, then  by the uniqueness of $f$ we have $\overline{f(z)}=f(\bar z)$, see e.g. \cite{ahlfors}, Exercise 1, p. 232 and so $f$ maps bijectively $G\cap\mathbb R$ onto $\mathbb D\cap\mathbb R$ and so
$f^{-1}$ is typically real. Conversely, assume that $f^{-1}:\mathbb D\to G$ is typically real. Then  $G$ is symmetric with respect to the real line.
\end{proof}
\begin{Cy}\label{Riemann intrinsic}
Let $G\subset\mathbb C$, $G$ nonempty, be a simply connected domain such that $G\cap\mathbb R\not=\emptyset$.
For a fixed $x_0\in G\cap \mathbb{R}$, there  exists a unique bijective,  analytic function
$f:\ G\to\mathbb D$ with $f(x_0)=0$, $f'(x_0)>0$; then $f$ is  such that $f^{-1}$ is typically real
 if and only if  $f$ is complex intrinsic.
\end{Cy}
\begin{proof}
If $f^{-1}$ is typically real then $G$ is symmetric with respect to the real line and $\overline{f^{-1}(w)}=f^{-1}(\bar w)$. By setting $w=f(z)$, i.e. $z=f^{-1}(w)$, we get
$$
f(\bar z)=f(\overline{f^{-1}(w)})=f(f^{-1}(\overline{w}))=\overline{w}=\overline{f(z)},
$$
and so $f$ is complex intrinsic. Conversely, let $f$ be complex intrinsic. It means that $f$ is defined on set $G$ symmetric with respect to the real line and $f(\bar z)=\overline{f(z)}$. Then $f^{-1}$ is complex intrinsic by the previous theorem.
\end{proof}
 \begin{Rk}\label{extension}{\rm Let $G\subset \mathbb{C}$ be symmetric with respect to the real axis and let the function $f$ as in the statement of the Riemann mapping theorem be intrinsic, i.e. $\overline{f(\bar z)}=f(z)$. By identifying $\mathbb C$ with $\mathbb C_J$ for some $J\in\mathbb S$ and using the extension formula in Corollary \ref{ext}, we obtain a function ${\rm ext}(f)(q)$ which is quaternionic intrinsic, in fact
\[
 \begin{split}
 {\rm ext}(f)(\bar q)&=\frac{1+IJ}{2} f(z) + \frac{1-IJ}{2} f(\bar z)
= \frac{1+IJ}{2} f(z) + \frac{1-IJ}{2} \overline {f(z)}\\
& ={\rm Re}f(z)+IJ^2 {\rm Im}f(z)=\overline{f(q)}.
\end{split}
\]
Thus ${\rm ext}(f)\in\mathcal N (\Omega_G)$ where $\Omega_G\subset \mathbb{H}$ denotes the axially symmetric completion of $G$.  For the sake of simplicity, we will simply write $f(q)$ instead of ${\rm ext}(f)(q)$.
}
\end{Rk}
\begin{Dn} We will denote by $\mathfrak{R}(\mathbb H)$ the class of  axially symmetric open sets $\Omega$ in $\mathbb H$ such that $\Omega\cap\mathbb C_I$ is simply connected for every $I\in\mathbb S$.
\end{Dn}
Note that $\Omega\cap\mathbb C_I$ is simply connected for every $I\in\mathbb S$ and thus it is connected, so $\Omega$ is an s-domain.\\
In view of Remark \ref{extension} and of Corollary \ref{Riemann intrinsic} we have the following :
\begin{Cy}
Let $\Omega\in \mathfrak{R}(\mathbb H)$,  $\mathbb{B}\subset \mathbb{H}$ be the open unit ball and let  $x_{0}\in \Omega\cap \mathbb{R}$.
Then there exists a unique quaternionic intrinsic slice regular function $f: \ \Omega \to \mathbb B$ which is bijective and such that $f(x_0)=0$, $f' (x_0)>0$.
\end{Cy}
\begin{proof}
Let us consider $\Omega_I=\Omega\cap\mathbb C_I$ where $I\in\mathbb S$. Then $\Omega_I$ is simply connected by hypothesis and symmetric with respect to the real line since $\Omega$ is axially symmetric. Let us set $\mathbb D_I=\mathbb B\cap\mathbb C_I$. By Corollary \ref{Riemann intrinsic} there exists a bijective, analytic intrinsic map $f_I: \Omega_I\to \mathbb D_I$ such that $f(x_0)=0$, $f'(x_0)>0$. By Remark \ref{extension}, $f_I$ extends to $f:\ \Omega\to \mathbb B$ and $f\in\mathcal N(\Omega)$. Note that for every $J\in\mathbb S$ we have $f_{|\mathbb C_J}: \Omega_J\to \mathbb D_J$ since $f$ takes each complex plane to itself.
\end{proof}
\begin{Rk}{\rm
The class $\mathfrak{R}(\mathbb H)$ contains all the possible simply connected open sets in $\mathbb H$ intersecting the real line for which a map $f: \Omega\to\mathbb B$ as in the Riemann mapping theorem belongs to the class $\mathcal{N}(\Omega)$. Indeed,  assume that a simply connected open set $\Omega\subset \mathbb H$ is mapped bijectively onto $\mathbb B$ with a map $f\in\mathcal{N}(\Omega)$, such that $f(x_0)=0$, $f'(x_0)>0$, $x_0\in\mathbb R$. Since $f\in\mathcal{N}(\Omega)$,  we have that  $f_{|\mathbb C_I}:\ \Omega\cap \mathbb C_I\to \mathbb B\cap\mathbb C_I=\mathbb D_I$  and so $f$ takes $\Omega\cap\mathbb R$ to $\mathbb B\cap\mathbb R$.  By its uniqueness, $f_{|\mathbb C_I}$  is the map prescribed by the complex Riemann mapping theorem, moreover $f_{|\mathbb C_I}$ takes $\Omega_I\cap\mathbb R$ bijectively to $\mathbb D_I\cap\mathbb R$ so $f^{-1}$ is totally real. By Corollary \ref{Riemann intrinsic}, it follows that $f$ is complex intrinsic, thus
$\Omega\cap \mathbb C_I$ is symmetric with respect to the real line.
 Since $I\in\mathbb S$ is arbitrary, $\Omega$ must be also axially symmetric, so it belongs to $\mathfrak{R}(\mathbb H)$.
}
\end{Rk}

\section{Approximation in some particular compacts}

For some particular cases of domains, the approximation by polynomials
can easily be proved, as for example in the following result.
\begin{theorem}\label{starlike}
Let $\Sigma\subset \mathbb{H}$ be a bounded region which is starlike with respect to the origin, and such that $\overline{\Sigma}$ is axially symmetric,
$\overline{\mathbb{H}}\setminus \overline{\Sigma}$ is connected and $\overline{\mathbb{C}_{I}}\setminus (\overline{\Sigma}\bigcap \mathbb{C}_{I})$ is connected for all $I\in \mathbb{S}$. If a function $f$ is slice regular in $\Sigma$ and continuous in $\overline\Sigma$ then the function $f$
can be uniformly approximated in $\overline{\Sigma}$ by polynomials.
\end{theorem}
\begin{proof}
Let us assume that $\Sigma$ is starlike with respect to the origin. Then the function $\varphi_n(q)=\frac{nq}{n+1}\in\mathcal N(\Sigma)$ for all $n\in\mathbb N$ and the composition $f(\varphi_n (q))=f(\frac{nq}{n+1})$ is slice regular in $\Sigma_n=\varphi_n^{-1}(\Sigma)$. The function $f(\frac{nq}{n+1})$ is defined in $\overline{\Sigma}$ and by the uniform continuity of $f$ in $\overline{\Sigma}$ we have that for any $\varepsilon >0$ and for a suitable $n$
$$
\left\|f\left(\frac{nq}{n+1}\right) -f(q)\right\| <\varepsilon /2
$$
for $q\in\overline{\Sigma}$. Then, since $f(\frac{nq}{n+1})$ is slice regular in $\overline{\Sigma}$, as a consequence of Runge's result in Theorem \ref{rungepoly}, there exists a polynomial $P(q)$ such that
$$
\left\|f\left(\frac{nq}{n+1}\right) -P(q)\right\| <\varepsilon /2
$$
for $q\in\overline{\Sigma}$, thus $\left\|f(q) -P(q)\right\| <\varepsilon$ and the statement follows.
\end{proof}

\begin{Rk}{\rm
 We now list some examples in the complex plane in e.g. \cite{Gal2}, pp. 82-83, Applications 1.9.8 and 1.9.9, a)-c), that we will use to present some concrete examples of compact sets in $\mathbb{H}$ where it is possible the approximation by polynomials.

{\bf Example 1.} Let $G\subset \mathbb{C}$ be bounded by the $m$-cusped hypocycloid $H_{m}$, $m=3, 4, ..., $ given by the parametric equation
$$z=e^{i \theta}+\frac{1}{m-1}e^{-(m-1)i \theta}, \, \theta\in [0, 2 \pi).$$
It is known that the conformal mapping (bijection) $\Psi : \overline{\mathbb{C}}\setminus \overline{\mathbb{D}}\to \overline{\mathbb{C}}\setminus G$ satisfying $\Psi(\infty)=\infty$ and $\Psi^{\prime}(\infty)>0$, is given by $\Psi(w)=w+\frac{1}{(m-1)w^{m-1}}$.

Note that the "original" Riemann mapping in the statement of Theorem 3.1 can be expressed with the aid of this mapping $\Psi$.

{\bf Example 2.} $G\subset \mathbb{C}$ is the $m$-leafed symmetric lemniscate, $m=2, 3, ..., $ with its boundary given by
$$L_{m}=\{z \in \mathbb{C} ; |z^{m}-1|=1\},$$
and the corresponding conformal mapping is given by $\Psi(w)=w \left (1+\frac{1}{w^{m}}\right )^{1/m}$.

{\bf Example 3.} $G\subset \mathbb{C}$ is the semidisk
$$SD=\{z\in \mathbb{C} ; |z|\le 1 \mbox{ and } |Arg(z)|\le \pi/2\},$$
with the corresponding conformal mapping given by $\Psi(w)=\frac{2(w^{3}-1)+3(w^{2}-w)+2(w^{2}+w+1)^{3/2}}{w(w+1)\sqrt{3}}$.

Given the set $G$ as in one of the above examples, we consider its axially symmetric completion $\Omega_G=\bigcup_{x+iy\in G,\ I\in\mathbb S} (x+Iy)$.
Note that $\overline\Omega_G$ is a starlike set with respect to the origin: for any $q=x+Iy\in\Omega_G$ consider $z=x+iy\in G$. Since $\bar G$ is starlike with respect to the origin in the complex plane $\mathbb C$,
the segment joining $z$ with the origin belongs to $\bar G$ and thus the segment joining $q$ with the origin is contained in $\overline\Omega_G$.
}
\end{Rk}

\begin{Cy}\label{cor_starlike}
Let $\Sigma$ be a bounded Jordan region which is starlike with respect to a real point, such that $\overline{\Sigma}$ is axially symmetric,
$\overline{\mathbb{H}}\setminus \overline{\Sigma}$ is connected and $\overline{\mathbb{C}_{I}}\setminus (\overline{\Sigma}\bigcap \mathbb{C}_{I})$ is connected for all $I\in \mathbb{S}$. Let $\alpha\in\overline{\Sigma}$. If a function $f$ is slice regular in $\Sigma$ and continuous in $\bar \Sigma$ then $f$
can be uniformly approximated in $\overline\Sigma$ by  polynomials which take the value $f(\alpha)$ at the point $\alpha$.
\end{Cy}
\begin{proof}
Let $\varepsilon >0$. By Theorem \ref{starlike} we know that there exists a polynomial $P(q)$ such that
$\left\|f(q) -P(q)\right\| <\varepsilon/2$ for $q\in\overline{\Sigma}$, so $\left\|f(\alpha) -P(\alpha)\right\| <\varepsilon/2$. By these two inequalities we have
$$
\left\|f(q) -(P(q)+f(\alpha) -P(\alpha))\right\| <\varepsilon
$$
for $q\in \overline{\Sigma}$ and thus the polynomial $P(q)+f(\alpha) -P(\alpha)$ satisfies the conditions in the statement.
\end{proof}
In what follows we obtain approximation results in the more general case of sets in $\mathfrak R(\mathbb H)$.

For that purpose, firstly we need a generalization of Theorem 4, p. 32 in \cite{walsh}.
\begin{theorem}\label{thm 4}
Let $T\in \mathfrak R(\mathbb H)$ be such that $T\cap \mathbb C_I$ is a bounded Jordan region for all $I\in \mathbb{S}$. Let $\{T_n\}$ be a sequence in $\mathfrak R(\mathbb H)$ such that $\overline{T}\subset T_n$, $\overline{T}_{n+1}\subset T_n$, for all $n=1,2,\ldots$ and no point exterior to $T$ belongs to $T_n$ for all $n$.
Let us assume that $q=0$ belongs to $T$ and that the bijective, quaternionic intrinsic functions $\Phi:\ T\to\mathbb B$, $\Phi_n:\ T_n\to\mathbb B$ all map $q=0$ to $w=0$ and $\Phi'(0)>0$,
$\Phi'_n(0)>0$ for all $n=1,2,\ldots$. Then:
\begin{equation}\label{Tuniform}
\lim_{n\to\infty} \Phi_n(q)=\Phi(q)
\end{equation}
uniformly for $ q\in\overline{T}$.
\end{theorem}
\begin{proof}
Let $I\in\mathbb S$ and let us consider $T\cap\mathbb C_I$, $T_n\cap\mathbb C_I$ and the restrictions $\Phi_{|\mathbb C_I}$, $\Phi_{n|\mathbb C_I}$ of $\Phi$, $\Phi_n$ respectively, to $\mathbb C_I$. Since $\Phi$, $\Phi_n$ are quaternionic intrinsic we have $\Phi_{|\mathbb C_I}:\ T\cap\mathbb C_I\to\mathbb D_I$, $\Phi_{n|\mathbb C_I}:\ T_n\cap\mathbb C_I\to\mathbb D_I$ and since all the hypothesis of Theorem 4 in \cite{walsh} are satisfied, we have that
\begin{equation}\label{uniform1}
\lim_{n\to\infty} \Phi_{n|\mathbb C_I}(z)=\Phi_{|\mathbb C_I}(z)
\end{equation}
uniformly for all $z\in \overline T\cap\mathbb C_I$. Let $K$ be a compact set in $\overline T$, then for any $q=x+Jy\in\ K$ we have by the Representation Formula (\ref{Repr_formula})
$$\|\Phi_n(q)-\Phi(q)\|$$
$$=\left\| \frac{1}{2}(1-JI) (\Phi_n(x+Iy)-\Phi(x+Iy))+\frac{1}{2}(1+JI) (\Phi_n(x-Iy)-\Phi(x-Iy))\right\|$$
$$
\leq \| \Phi_n(x+Iy)-\Phi(x+Iy))\|+\|\Phi_n(x-Iy)-\Phi(x-Iy)\|.
$$
Now note that the right hand side corresponds to the restriction of $\Phi_n$, $\Phi$ to $K\cap\mathbb C_I$ which is a compact subset of $\overline T\cap \mathbb C_I$ and thus, by (\ref{uniform1}) we have
$$
 \| \Phi_n(x+Iy)-\Phi(x+Iy))\|+\|\Phi_n(x-Iy)-\Phi(x-Iy)\|<\varepsilon,
 $$
for $n>N(\varepsilon)$, which concludes the proof.
\end{proof}
\begin{Cy}\label{corthem4}
Under the hypotheses of Theorem \ref{thm 4}, let $\chi_n:\ T_n\to T$ be defined by $\chi_n(q)=\Phi^{-1}(\Phi_n(q))$. Then $\chi_n(0)=0$, $\chi'_n(0)>0$ and
$$
\lim_{n\to\infty}\chi_n(q)=q
$$
uniformly in $\overline T$.
\end{Cy}
\begin{proof}
The function $\Phi^{-1}$ is continuous in $\overline T$ thus we can apply it to both members of (\ref{Tuniform}) and we get the statement.
\end{proof}
The following result further generalizes Theorem \ref{starlike}:
\begin{theorem}\label{thm 5} Let $T\in\mathfrak R(\mathbb H)$ be bounded and such that $T\cap\mathbb C_I$ is a Jordan region in the plane $\mathbb C_I$, for all $I\in \mathbb{S}$.
If $f$ is slice regular in $T$, continuous in $\overline T$, then in $\overline T$ the function $f(q)$ can be uniformly approximated by polynomials in $q$.
\end{theorem}
\begin{proof}
Consider a sequence $\{T_n\}$ and the maps $\Phi_n$, $\Phi$ as in the statement of Theorem \ref{thm 4}. Let $\chi_n:\ T_n\to T$ be defined by $\chi_n(q)=\Phi^{-1}(\Phi_n(q))$. Then $\chi_n\in\mathcal N(T_n)$ thus we can consider the composition $f(\chi_n(q))$ which is slice regular in $\overline T$.
By Corollary \ref{corthem4} and the continuity of $f$, for any $\varepsilon >0$ there exists $N(\varepsilon)$ such that
$$
\| f(q)-f(\chi_n(q))\|<\varepsilon/2
$$
for $q\in\overline T$. By Theorem \ref{rungepoly} there exists a polynomial $P(q)$ such that
$$
\| f(\chi_n(q))-P(q)\|<\varepsilon/2
$$
for $q\in\overline T$, thus we deduce that
$$
\| f(q)- P(q)\|<\varepsilon
$$
for $q\in\overline T$ and this concludes the proof.
\end{proof}
The proof of the following corollary follows exactly the proof of Corollary \ref{cor_starlike}, so we will omit it.
\begin{Cy}
Let $T\in\mathfrak R(\mathbb H)$ be bounded and such that $T\cap\mathbb C_I$ is a Jordan region in the plane $\mathbb C_I$.
 Let $\alpha\in\overline T$. If a function $f$ is slice regular in $T$ and continuous in $\overline T$ then $f$
can be uniformly approximated in $\overline T$ by polynomials which take the value $f(\alpha)$ at the point $\alpha$.
\end{Cy}

In what follows, we want to discuss approximation results by polynomials in two cases: for compact balls with centers at real numbers and for the so called
Cassini cells. The proofs of these results are interesting from two points of views:
\begin{enumerate}
\item[a)] they are completely constructive being based on convolutions with some well-known trigonometric kernel;
 \item[b)] they allow, in addition, to obtain quantitative estimate in terms of the modulus of continuity (fact which does not happen in the previous approximation results).
\end{enumerate}
In order to state and prove the results, we need to recall some preliminary facts.\\
As in \cite{GaSa}, consider the classical de la Vall\'ee Poussin kernel given by
$$K_{n}(u)=\frac{(n!)^{2}}{(2n)!}\left (2\cos\frac{u}{2}\right )^{2n}=1+2\sum_{j=1}^{n}\frac{(n!)^{2}}{(n-j)!(n+j)!}\cos(j u), u\in \mathbb{R}$$
and for $f:\overline{B(0;R)}\to \mathbb{H}$, left slice regular on $B(0;R)$, i.e. $f(q)=\sum_{k=0}^{\infty}q^{k}c_{k}$, $q\in B(0;R)$, and $f$ continuous on $\overline{B(0;R)}$, define the convolution operator of quaternion variable
$$T_{n, l}(f)(q)=\frac{1}{2\pi}\int_{-\pi}^{\pi}f(q e^{I_{q} u}) K_{n}(u)d u, \, q\in \mathbb{H}\setminus \mathbb{R}, \quad q=re^{I_{q} t}\in \overline{B(0;R)}$$
$$T_{n, l}(f)(q)=\frac{1}{2\pi}\int_{-\pi}^{\pi}f(q e^{I u}) K_{n}(u)d u, \, q\in \mathbb{R}\setminus \{0\}, \qquad q=\|q\|e^{I t}\in \overline{B(0;R)},  t=0\ {\rm or}\ \pi,$$
where $I$ is an arbitrary element in $\mathbb S$, and
$$T_{n, l}(f)(0)=\frac{1}{2\pi}f(0)\int_{-\pi}^{\pi}K_{n}(u)d u=f(0).$$
By formula (3.1) in \cite{GaSa} we can write
\begin{equation}\label{polyn_ball}
T_{n, l}(f)(q):=P_{n}(q)=\sum_{l=0}^{n}q^{l}c_{l}\cdot \frac{(n!)^{2}}{(n-l)!(n+l)!},\quad q\in {B}(0,{R}).
\end{equation}
\begin{theorem}({\bf Quantitative approximation on compact balls})\label{appr_balls}
Let $R>0$,  $B(0;R)=\{q\in \mathbb{H} ; \|q\| < R\}$, $K=\overline{B(0;R)}$  and $f:K\to \mathbb{H}$ be continuous on $K$ and slice regular on the interior of $K$, $int(K)=B(0;R)$.
Then for any $\varepsilon >0$ there exists a polynomial $P$ such that
$\|f(q)-P(q)\|<\varepsilon$ for all $q\in K$.

In fact, for all  $q\in \overline{B(0;R)}$ and $n\in \mathbb{N}$ we have
$$|P_{n}(q)-f(q)|\le 3(R+1)\omega_{1}(f; 1/\sqrt{n}),$$
where $\omega_{1}(f; \delta)=\sup\{\|f(u)-f(v)\|; u, v\in \overline{B(0;R)}, \|u-v\|\le \delta\}$ and $P_{n}(q)$ are the polynomials given by (\ref{polyn_ball}).
\end{theorem}

\begin{proof} For any $q\in \overline{B(0;R)}$ and reasoning exactly as in the complex case in \cite{Ga}, p. 427, we easily get
$$\|T_{n, l}(f)(q)-f(q)\|\le \frac{1}{2\pi}\int_{-\pi}^{\pi}\|f(q e^{I_{q} u}) - f(q)\|K_{n}(u)d u$$
$$\le \frac{1}{2\pi}\int_{-\pi}^{\pi}\omega_{1}(f ; \|q\|\cdot \|e^{I_{q} u}-1\|)K_{n}(u)d u\le \frac{1}{2\pi}\int_{-\pi}^{\pi}\omega_{1}(f ; R |u|)K_{n}(u)d u$$
$$\le (R+1)\frac{1}{2\pi}\int_{-\pi}^{\pi}\omega_{1}\left (f ; \frac{1}{\sqrt{n}} |u|\sqrt{n}\right )K_{n}(u)d u$$
$$\le (R+1)\omega_{1}(f ; 1/\sqrt{n})\frac{1}{2\pi}\int_{-\pi}^{\pi}(|u|\sqrt{n}+1)K_{n}(u) dt\le 3(R+1)\omega_{1}(f ; 1/\sqrt{n}),$$
where $\omega_{1}(f; \delta)=\sup\{\|f(u)-f(v)\|; u, v\in \overline{B(0;R)}, \|u-v\|\le \delta\}$.

Let $\varepsilon >0$ be arbitrary. By the continuity of $f$ on $\overline{B(0;R)}$ we have $\lim_{n\to \infty}\omega_{1}(f ;1/\sqrt{n})=0$,
which means that there exists $n_{0}$ such that for all $n\ge n_{0}$ we have
$$\|T_{n, l}(f)(q)-f(q)\|\le \varepsilon, \mbox{ for all } q\in \overline{B(0;R)}.$$
But by the relationship (\ref{polyn_ball}), if $q\in B(0;R)$ then $T_{n, l}(f)(q)$ is a right polynomial, $P_{n}(q)$, i.e. we have
\begin{equation}\label{eq:approx}
\|P_{n}(q)-f(q)\|\le \varepsilon, \mbox{ for all } q\in B(0;R).
\end{equation}
Now, let $q_{0}\in \overline{B(0;R)}$. There exists a sequence $q_{m}\in B(0;R)$, $m\in \mathbb{N}$, such that $\|q_{m}-q_{0}\|\to 0$ as
$m\to \infty$. Replacing in (\ref{eq:approx}) $q$ by $q_{m}$ and then passing to limit with $m\to \infty$, by the continuity of $P_{n}$ and $f$ it follows
that $\|P_{n}(q_{0})-f(q_{0})\|\le \varepsilon$, which proves the theorem.
\end{proof}

\begin{Rk}{\rm
\begin{enumerate}
\item We have a little more than in the classical Mergelyan's result : we get a quantitative estimate
in terms of the modulus of continuity $\omega_{1}$.

 \item Theorem \ref{appr_balls} can be easily extended to functions defined on the balls $\overline{B(x_{0} ; R)}$, where $x_{0}\in \mathbb{R}$, by taking $Q_{n}(q)=P_{n}(q-x_{0})$, $q\in \overline{B(x_{0};R)}$, with $P_{n}(q)$ given in the proof of Theorem \ref{appr_balls}.

 \item If $x_{0}\in \mathbb{H}\setminus \mathbb{R}$, then the above Point 2 does not hold. However, by using a new development in series in \cite{Sto}, in what follows we will extend Theorem \ref{appr_balls} to more general compact sets than the closed balls with real numbers as centers.
\end{enumerate}
}
\end{Rk}
For $q_{0}=x_{0}+I y_{0}\in \mathbb{H}$, with $x_{0}, y_{0}\in \mathbb{R}$, $y_{0}>0$, $I\in \mathbb{S}$ and $R>0$, let us denote
$$B(x_{0}+y_{0}\mathbb{S} ; R)=\{q\in \mathbb{H} ; \|(q-x_{0})^{2}+y_{0}^{2}\| < R^{2}\}.$$
A set of the form $B(x_{0}+y_{0}\mathbb{S} ; R)$ will be called Cassini cell (it is defined by means of the so-called Cassini pseudo-metric, see \cite{gp}).  If $y_{0}=0$, then clearly we have $B(x_{0}+\mathbb{S}y_{0} ; R)=B(x_{0};R)$.

According to Theorem 5.4 in \cite{gp}, if $f:B(x_{0}+y_{0}\mathbb{S} ; R)\to \mathbb{H}$ is slice regular then it admits the representation
\begin{equation}\label{develop}
f(q)=\sum_{k=1}^{\infty}[(q-x_{0})^{2}+y_{0}^{2}]^{k}[c_{2k}+q c_{2k+1}], \, \mbox{ for all } q\in B(x_{0}+y_{0}\mathbb{S} ; R),
\end{equation}
where $c_{2k}, c_{2k+1}\in \mathbb{H}$, for all $k\in \mathbb{N}$.

Similar to the case of Theorem \ref{appr_balls}, for functions $f$ of the above form, let us define the convolution
$$V_{n, l}(f)(q)=\frac{1}{2\pi}\int_{-\pi}^{\pi}f(q e^{I_{q} u}) K_{n}(u)d u, q\in \mathbb{H}\setminus \mathbb{R}, \quad q=re^{I_{q} t}\in \overline{B(x_{0}+y_{0}\mathbb{S} ; R)}$$
$$V_{n, l}(f)(q)=\frac{1}{2\pi}\int_{-\pi}^{\pi}f(q e^{I u}) K_{n}(u)d u, \, q\in \mathbb{R}\setminus \{0\}, q=\|q\|e^{I t}\in \overline{B(x_{0}+y_{0}\mathbb{S};R)},  t=0\ {\rm or}\ \pi,$$
where $I$ is an arbitrary element in $\mathbb S$, and
$$V_{n, l}(f)(0)=\frac{1}{2\pi}f(0)\int_{-\pi}^{\pi}K_{n}(u)d u=f(0).$$
We have :

\begin{theorem} ({\bf Quantitative approximation on Cassini cells})\label{Cassini}
Let $q_{0}=x_{0}+I y_{0}\in \mathbb{H}$, with $x_{0}, y_{0}\in \mathbb{R}$, $y_{0}>0$, $I\in \mathbb{S}$ and $R>0$. If $f:\overline{B(x_{0}+y_{0}\mathbb{S};R)}\to \mathbb{H}$ is continuous in $\overline{B(x_{0}+y_{0}\mathbb{S};R)}$ and slice regular in
$B(x_{0}+y_{0}\mathbb{S};R)$, then for any $\varepsilon >0$ there exists a polynomial $P$ such that
$\|f(q)-P(q)\|<\varepsilon$ for all $q\in \overline{B(x_{0}+y_{0}\mathbb{S};R)}$.

In fact, for all  $q\in \overline{B(x_{0}+y_{0}\mathbb{S};R)}$ and $n\in \mathbb{N}$ we have
$$|V_{n, l}(f)(q)-f(q)|\le 3(R+1)\omega_{1}(f; 1/\sqrt{n}),$$
where $V_{n, l}(f)(q)$ is a polynomial.
\end{theorem}

{\bf Proof.} Reasoning as in the proof of Theorem \ref{appr_balls}, for any $q\in \overline{B(x_{0}+y_{0}\mathbb{S};R)}$ we obtain
$$\|V_{n, l}(f)(q)-f(q)\|\le \frac{1}{2\pi}\int_{-\pi}^{\pi}\|f(q e^{I_{q} u}) - f(q)\|K_{n}(u)d u$$
$$\le \frac{1}{2\pi}\int_{-\pi}^{\pi}\omega_{1}(f ; \|q\|\cdot \|e^{I_{q} u}-1\|)K_{n}(u)d u\le \frac{1}{2\pi}\int_{-\pi}^{\pi}\omega_{1}(f ; (M_{R, q_{0}})|u|)K_{n}(u)d u$$
$$\le (M_{R, q_{0}}+1)\frac{1}{2\pi}\int_{-\pi}^{\pi}\omega_{1}\left (f ; \frac{1}{\sqrt{n}} |u|\sqrt{n}\right )K_{n}(u)d u$$
$$\le 3(M_{R, q_{0}}+1)\omega_{1}(f ; 1/\sqrt{n}).$$
We have used above that if $q\in \overline{B(x_{0}+y_{0}\mathbb{S};R)}$, then
$$\|q\|^{2}=\|q^{2}\|=\|(q-x_{0}+x_{0})^{2}+y_{0}^{2}-y_{0}^{2}\|=\|(q-x_{0})^{2}+y_{0}^{2}+2x_{0}(q-x_{0})+x_{0}^{2}-y_{0}^{2}\|$$
$$\le \|(q-x_{0})^{2}+y_{0}^{2}\|+2|x_{0}|\cdot \|q-x_{0}\|+|x_{0}^{2}+y_{0}^{2}|$$
$$\le R^{2}+2|x_{0}|\cdot \sqrt{R^{2}+y_{0}^{2}}+\|q_{0}\|^{2}:=M_{R, q_{0}}^{2},$$
which implies $\|q\|\le M_{R, q_{0}}$.

What was remained to prove is that $V_{n, l}(f)(q)$ is a polynomial.

Indeed, by (\ref{develop}) we can write
$$f(q)=\sum_{k=1}^{\infty}\left (\sum_{j=0}^{k}{k \choose j}(q-x_{0})^{2j}\cdot y_{0}^{2(k-j)}\right )c_{2k}+
\sum_{k=1}^{\infty}\left (\sum_{j=0}^{k}{k \choose j}(q-x_{0})^{2j}\cdot q \cdot y_{0}^{2(k-j)}\right )c_{2k+1}$$
$$=\sum_{k=1}^{\infty}\left (\sum_{j=0}^{k}{k \choose j}y_{0}^{2(k-j)}\sum_{p=0}^{2j}{2j \choose p}q^{p}(-1)^{2j-p }x_{0}^{2j-p}\right )c_{2k}$$
$$+\sum_{k=1}^{\infty}\left (\sum_{j=0}^{k}{k \choose j}y_{0}^{2(k-j)}\sum_{p=0}^{2j}{2j \choose p}q^{p}(-1)^{2j-p }x_{0}^{2j-p}\right )q c_{2k+1}                                                 ,$$
which implies
$$V_{n, l}(f)(q)=\frac{1}{2\pi}\int_{-\pi}^{\pi}f(q e^{I_{q} u})K_{n}(u)d u$$
$$=\sum_{k=1}^{\infty}\left (\sum_{j=0}^{k}{k \choose j}y_{0}^{2(k-j)}\sum_{p=0}^{2j}{2j \choose p}q^{p}\left [\frac{1}{2\pi}\int_{-\pi}^{\pi}e^{I_{q} p u}K_{n}(u)du\right ](-1)^{2j-p }x_{0}^{2j-p}\right )c_{2k}$$
$$+\sum_{k=1}^{\infty}\left (\sum_{j=0}^{k}{k \choose j}y_{0}^{2(k-j)}\sum_{p=0}^{2j}{2j \choose p}q^{p}\left [\frac{1}{2\pi}\int_{-\pi}^{\pi}e^{I_{q} p u}K_{n}(u)du\right ](-1)^{2j-p }x_{0}^{2j-p}\right )qc_{2k+1}.$$
But
$$\int_{-\pi}^{\pi}e^{I_{q} p u}K_{n}(u)d u=\int_{-\pi}^{\pi}e^{I_{q}p u}\left (1+\sum_{s=1}^{n}\frac{(n!)^{2}}{(n-s)!(n+s)!}(e^{I_{q} s u}+e^{-I_{q}s u})\right )du$$
$$=\int_{-\pi}^{\pi}e^{I_{q}pu}du+\sum_{s=1}^{n}\frac{(n!)^{2}}{(n-s)!(n+s)!}\left (\int_{-\pi}^{\pi}e^{I_{q}(p+s)u}du+\int_{-\pi}^{\pi}e^{I_{q}(p-s)u}du\right ).$$
Now, taking into account that $\int_{-\pi}^{\pi}e^{I_{q} p u}d u=0$ if $p\not=0$, $\int_{-\pi}^{\pi}e^{I_{q} p u}d u=2\pi$ if $p=0$,
$\int_{-\pi}^{\pi}e^{I_{q}(p+s)u}du=0$ for all $s\ge 1$ and $p\ge 0$ and that $\int_{-\pi}^{\pi}e^{I_{q}(p-s)u}du\not =0$ only for $p=s$, since $s\in \{1, ..., n\}$ it immediately follows that the infinite sum which express $V_{n, l}(f)(q)$ reduces to a finite sum containing powers of $q$ less than $n$ and with all the coefficients written on the right. In conclusion, $V_{n, l}(f)(q)$ is a polynomial, which ends the proof of the theorem. $\hfill \square$

\begin{Rk}{\rm
\begin{enumerate}
\item If in the definitions of $T_{n, l}(f)(q)$ and $V_{n, l}(f)(q)$ and in the statements of Theorems \ref{appr_balls} and \ref{Cassini}, instead of the de la Vall\'ee-Poussin kernel $\frac{1}{2\pi}K_{n}(u)$, we use the Jackson kernel $J_{n}(u)=\frac{1}{\pi}\cdot \frac{3}{2n(2n^{2}+1)}\left (\frac{\sin(n u/2)}{\sin(u/2)}\right )^{4}$, then reasoning as in the complex case in \cite{Ga}, pp. 422-423, we get the better quantitative estimate in terms of $\omega_{2}(f ; 1/n)$.

 \item More in general, if instead of the polynomial operators $T_{n, l}(f)(q)$ and $V_{n, l}(f)(q)$ in the Theorems \ref{appr_balls} and \ref{Cassini}, for any fixed $p\in \mathbb{N}$ we consider the polynomial operators of the form
     $$L_{n, l, p}(f)(q)=-\int_{-\pi}^{\pi}K_{n, r}(u)\sum_{k=1}^{p+1}{p+1 \choose k}f(qe^{I_{q}k u})du,$$
     where $r$ is the smallest integer for which $r\ge (p+2)/2$ and $K_{n, r}(u)=\frac{1}{\lambda_{n, r}}\left (\frac{\sin(n u/2)}{\sin(u/2)}\right )^{2r}$  is the normalized generalized Jackson kernel, then reasoning as in the case of complex variable in \cite{Ga}, p. 424, better quantitative estimates in terms of $\omega_{p+1}(f ; 1/n)$ are obtained.

 \item In the most general case, denoting by $F_{n}(u)=\frac{1}{2}\left (\frac{\sin(n u/2)}{\sin(u/2)}\right )^{2}$ the Fej\'er kernel and
$L_{n}(f)(q)=\frac{1}{n\pi}\int_{-\pi}^{\pi}F_{n}(u)f(qe^{I_{q} u})d u$, defining the new polynomial operators $P_{n}(f)(q)=2L_{2 n}(f)(q)-L_{n}(f)(q)$, similar reasonings with those in the complex case in \cite{Ga}, pp. 424-425 allow us to recapture the statements of Theorems \ref{appr_balls} and \ref{Cassini} with quantitative estimates in terms of $E_{n}(f)$, where $E_{n}(f)=\inf\{\||p-f\|| ; p \mbox{ polynomial of degree } \le n \}$ is the best approximation of $f$ by polynomials of degree $\le n$ and $\||f\||=\sup_{q}\{\|f(q)\|\}$.
\end{enumerate}
}
\end{Rk}

\begin{Rk}\label{Weierstrass}
{\rm If one renounce to the regularity and we suppose only the continuity of $f$ in a compact set of $\mathbb{C}$ (say, for example, the closed unit ball), it is well-known that in general $f$ cannot be uniformly approximated by polynomials. However, in this case, we can prove that $f$ can be uniformly approximated by polynomials in $z$ and $z^{-1}$.
This can be done by virtue of the Weierstrass theorem on approximation by trigonometric polynomials that claims that  a real valued function $g(\theta)$ which is continuous for all $\theta$ and periodic of period $2\pi$ can be uniformly approximated  for all $\theta$ by a trigonometric polynomial of the form $\sum_{n=0}^N \cos(n \theta) a_n + \sin(n \theta) b_n$. It is immediate that if the function $g$ is complex (resp. quaternionic) valued the theorem holds with $a_n,b_n\in\mathbb C$ (resp. $\mathbb H$).
Let us now consider a continuous slice function (see Definition \ref{slice function}) $f(q)=f(x+Iy)=\alpha(x,y)+I\beta(x,y)$. Take $q\in\mathbb H$ such that $\|q\|=1$ and write $q^n=\cos(n\theta)+I\sin(n\theta)$, $q^{-n}=\cos(n\theta)-I\sin(n\theta)$. Then $\cos(n\theta)=\frac 12 (q^n+q^{-n})$, $\sin(n\theta)=\frac 12 I (q^{-n}-q^n)$. Moreover, $f(q)=f(x+Iy)=f(\cos\theta +I\sin\theta)=\alpha(\theta)+I\beta(\theta)$ where $\alpha$, $\beta$ are $\mathbb H$-valued.}
\end{Rk}
As a consequence of this discussion, we present the generalization of Theorem 7, $\S$ 2.5 in \cite{walsh}:
\begin{theorem}
Let $\Sigma$ be the boundary of an open set belonging to  $\mathfrak{R}(\mathbb H)$ and containing the origin and let us assume that $\Sigma\cap\mathbb{C}_I$ is a Jordan curve for any $I\in\mathbb S$. Then a continuous slice function $f$ on $\Sigma$ can be uniformly approximated on $\Sigma$ by polynomials in $q$ and $q^{-1}$.
\end{theorem}
\begin{proof}
By hypothesis the interior $\Omega$ of $\Sigma$ can be mapped onto the unit ball $\mathbb B$ of $\mathbb H$ by a function $\Phi\in\mathcal N(\Omega)$. Let $\Psi$ be the inverse of $\Phi$, i.e. $q=\Psi (w)$, $w\in\mathbb B$. Now observe that by the above mentioned Weierstrass approximation theorem by trigonometric polynomials, for any $\varepsilon >0$ there exists a polynomial $P(w,1/w)$ such that
$$
\| f(\Psi(w))- P(w,1/w)\| <\varepsilon /2, \qquad  w\in \partial\mathbb B\cap \mathbb C_J,
$$
in fact,  $z=\Psi(w)=\cos(\theta)+J\sin(\theta)$ if and only if $w\in \mathbb C_J$. The inequality can be written as
$$
\| f(z)- P(\Phi(z),\Phi(z)^{-1})\| <\varepsilon /2, \qquad z \in \Sigma \cap \mathbb C_J.
$$
Let us  now extend the polynomial $P$ from $\mathbb C_J$ to $\mathbb H$ using the extension formula for slice continuous functions, see \cite{gp1}, (and recall that $\overline{\Phi(z)}=\Phi(\bar z)$):
$$
P(\Phi(q),\Phi(q)^{-1})=\frac{1-IJ}{2}P(\Phi(z),\Phi(z)^{-1}) +\frac{1+IJ}{2}P(\Phi(\bar z),\Phi(\bar z)^{-1})
$$
where $q=x+Iy$ and so, for any $q\in\Sigma$ we have:
\[
\begin{split}
&\left \| f(q)- P(\Phi(q),\Phi(q)^{-1})\right \|\\
&=\left \|   \frac{1-IJ}{2}f(z) +\frac{1+IJ}{2}f(\bar z) - \frac{1-IJ}{2}P(\Phi(z),\Phi(z)^{-1}) -\frac{1+IJ}{2}P(\Phi(\bar z),\Phi(\bar z)^{-1}) \right \| \\
\end{split}
\]
$$
\leq \left \|   \frac{1-IJ}{2} \right \|\cdot  \left \| f(z) -P(\Phi(z),\Phi(z)^{-1})\right \|  + \left \|\frac{1+IJ}{2}\right \|\cdot  \left \| f(\bar z) - P(\Phi(\bar z),\Phi(\bar z)^{-1}) \right \| <\varepsilon.
$$
The function $P(\Phi(q),\Phi(q)^{-1})$ is slice regular in the region $R$ bounded by $\Sigma$ and by the boundary of a ball with center at the origin.
By Theorem 4.10 in \cite{runge} there exists a rational function $Q$ with poles at zero and at infinity, i.e. a polynomial $Q(q,q^{-1})$ in $q$, $q^{-1}$ such that
$$
\|  P(\Phi(q),\Phi(q)^{-1})- Q(q,q^{-1}) \|< \varepsilon /2
$$
for $q\in\overline R$. By taking $q\in\Sigma$ we obtain, in particular, that
$$
\| f(q)-  Q(q,q^{-1}) \|< \varepsilon
$$
which proves the statement.
\end{proof}

{\bf Acknowledgements}. We thank Maxime Fortier Bourque, from the Department of Mathematics, City University of New York, for pointing out to us Theorem \ref{Tm:Riemann intrinsic}.

\end{document}